\documentclass[11pt]{article}

\usepackage[letterpaper,,margin=1in]{geometry}
\usepackage{graphicx} 
\usepackage{amsmath,amssymb,amsthm}
\usepackage{xcolor}
\usepackage{enumitem}
\usepackage{mathtools}

\usepackage[square,numbers]{natbib}

\usepackage{url}            
\usepackage[colorlinks=true,citecolor=blue]{hyperref}
\usepackage{cleveref}

\newtheorem{theorem}{Theorem}

\newtheorem{lemma}{Lemma}
\newtheorem{proposition}{Proposition}
\newtheorem{fact}{Fact}

\theoremstyle{definition}

\newcommand{\norm}[1]{\ensuremath{\left\lVert #1 \right\rVert}}
\newcommand{\ip}[1]{\ensuremath{\left\langle #1 \right\rangle}}

\let\emptyset\varnothing
\newcommand{\set}[1]{\left\{#1\right\}}

\def\L{{\mathbb{L}}}

\def\R{{\mathbb{R}}}
\def\S{{\mathbb{S}}}

\def\cA{{\cal A}}

\def\cC{{\cal C}}
\def\cD{{\cal D}}

\def\cK{{\cal K}}

\def\cM{{\cal M}}

\def\cV{{\cal V}}

\def\cX{{\cal X}}

\def\cZ{{\cal Z}}

\DeclarePairedDelimiterX{\inner}[2]{\langle}{\rangle}{#1, #2}

\DeclareMathOperator{\rank}{rank}

\DeclareMathOperator{\tr}{tr}

\DeclareMathOperator{\bd}{bd}

\DeclareMathOperator{\conv}{conv}
\DeclareMathOperator{\cone}{cone}

\DeclareMathOperator{\Proj}{Proj}

\renewcommand{\b}{{\bm b}}

\renewcommand{\t}{{\bm t}}

\newcommand{\Zb}{\bar{Z}}

\begin{document}

\title{On the strength of Burer's lifted convex relaxation to quadratic programming with ball constraints}

\author{Fatma K{\i}l{\i}n{\c{c}}-Karzan, Shengding Sun}
\date{\today}

\maketitle

\begin{abstract}

We study quadratic programs with $m$ ball constraints, and the strength of a lifted convex relaxation for it recently proposed by Burer (2024). Burer shows this relaxation is exact when $m=2$. For general $m$, Burer (2024) provides numerical evidence that this lifted relaxation is tighter than the Kronecker product based \emph{Reformulation Linearization Technique} (RLT) inequalities introduced by Anstreicher (2017), and conjectures that this must be theoretically true as well. In this note, we provide an affirmative answer to this question and formally prove that this lifted relaxation indeed implies the Kronecker inequalities. 
Our proof is based on a decomposition of non-rank-one extreme rays of the lifted relaxation for each pair of ball constraints. Burer (2024) also numerically observes that for this lifted relaxation, an RLT-based inequality proposed by Zhen et al.\ (2021) is redundant, and conjectures this to be theoretically true as well. 
We also provide a formal proof that Zhen et al.\ (2021) inequalities are redundant for this lifted relaxation. In addition, we establish that Burer's  lifted relaxation is a particular case of the moment-sum-of-squares hierarchy. 
\end{abstract}

\section{Introduction}

In this note, we study a specific class of quadratically constrained quadratic programs (QCQPs) with a general quadratic objective and $m\ge 2$ ball constraints:
\begin{align}\label{eq:prob}
&\min_{x\in\R^n} \set{q(x):~ x\in S}, \quad
\text{where }~ S:=\set{x\in\R^n:~ \|x-c^i\|_2\le r_i,~\forall i\in [m] } .
\end{align}
Here, $[m]:=\{1,\ldots,m\}$, and  $c^i\in\R^n$ and $r_i\in\R$ for all $i\in[m]$, and $\|\cdot\|_2$ stands for the Euclidean norm. We assume $S\ne \emptyset$.

QCQP has a long history in the optimization community, dating back to 1990s. For fixed $m$, QCQPs with domain $S$ given by the intersection of $m$ ball constraints can be solved within desired accuracy in polynomial time \cite{bienstock2016note}. However, no exact convex relaxation is known for all $m$. 

A closely related direction is to study the conical hull of lifted quadratic set of $S$ given by
\begin{equation}\label{eqn:exact}
    \cD:=\cone \{(xx^\top,x):~ x\in S\}\subseteq \S^n_+ \oplus \R^n, 
\end{equation}
where $\cone\{\cX \}$ denotes the conical hull of the set $\cX$ (i.e., set of all nonnegative combinations of the vectors from $\cX$), $\S^n$ is the space of $n\times n$ symmetric matrices and $\S^n_+$ is the cone of positive semidefinite matrices in $\S^n$.

Clearly any (possibly non-convex) quadratic objective over $S$ can be equivalently written as a linear objective over $\cD$. Thus, whenever $\cD$ admits an explicit and concise description, problem~\eqref{eq:prob}  can be efficiently solved. 

The set $\cD$ immediately brings to our attention the  well known Shor semidefinite programming (SDP) relaxation \cite{shor1990dual}, which is computationally efficient but in general only a relaxation.  
Shor SDP relaxation has been studied a lot in the literature and several efforts have been made to tighten it (see \cite{bao2011semidefinite,burer2013second,burer2015gentle} and other references therein). One of the latest such effort is the Kronecker product based \emph{Reformulation Linearization Technique} (RLT) inequalities proposed by \citet{anstreicher2017kronecker}, which are applicable when $S$ is defined by second-order cone (SOC) constraints. When $S$ is defined by SOC constraints,  \citet{zhen2021extension} propose another RLT-based inequality that aims to strengthen the Shor relaxation. 

There has been much work \cite{sturm2003cones,burer2015trust,yang2018quadratic} on studying exact or tight description of $\cD$ for certain classes of $S$. Recently,  when $S$ is given by the intersection of two ball constraints, \citet{kelly2022note} gave an exact disjunctive SDP reformulation. Inspired by this result, for arbitrary number of ball constraints, \citet{burer2024slightly} proposed a lifted  convex relaxation which involves only one more additional variable for lifting, and showed that although it is not an exact relaxation for general $m>2$, its admits a very good numerical performance in terms of both relaxation quality and efficiency of computation. 
In \cite{burer2024slightly}, it was conjectured that this lifted relaxation is provably tighter than the Kronecker RLT inequalities and Zhen et al.'s RLT inequalities are redundant for this lifted relaxation. We answer these two conjectures affirmatively. In particular, in Theorem~\ref{thm:main}, by proving a decomposition theorem (see Theorem~\ref{thm:non_rank_one_extray}) for the  non-rank-one extreme rays of the lifted relaxation proposed in \cite{burer2024slightly} and studying the properties of both rank-one and non-rank-one extreme rays, we show that Kronecker RLT inequalities are redundant for the projection of this lifted relaxation.  In Theorem~\ref{thm:Zhen} we show that Zhen et al.\ RLT inequalities are implied by this lifted relaxation as well. Finally, we close by giving a new interpretation of Burer's lifted relaxation using the techniques from moment-sum-of-squares (moment-SOS) hierarchy. 

\section{Burer's lifted convex relaxation}
\citet{burer2024slightly} proposes a lifted convex relaxation of $\cD$ in the space $\S^{n+2}$, where recall $S$ is defined by $m$ ball constraints 
\[
S:=\set{x\in\R^n:~ \|x-c^i\|_2\le r_i,~\forall i\in [m] } ,
\]
and
\[
\cD:= \cone\{(xx^{\top},x):~ x\in S\}. 
\]
We denote the second-order (SOC) cone (also known as Lorentz cone) in $\R^{n+2}$ by
    \[
    \L^{n+2}
    := \set{x\in\R^{n+2}:~ \norm{(x_1,\ldots,x_{n+1})}_2\le x_{n+2}} ,
    \]    
and denoting $I$ as the $n\times n$ identity matrix, we define
    \[
Q:=\begin{bmatrix}-2I & 0 & 0\\0 & 0 & 1\\0 & 1 & 0\end{bmatrix},\quad
P:=\begin{bmatrix}2I & 0 & 0\\ 0 & 1 & -1\\0 & 1 & 1
\end{bmatrix},\quad
d^i:=\begin{bmatrix}2c^i\\-1\\r_i^2-\|c^i\|_2^2\end{bmatrix},~\forall i\in[m].
\]

Then Burer's lifted relaxation\footnote{Upon conjugation by a permutation matrix} is the intersection of the \emph{convex} set given by
    \begin{equation}\label{eqn:burer}
    \cC_m :=\set{Z\in\S^{n+2}_+:\, \ip{Q,Z} = 0,
    \, PZd^i\in\L^{n+2},\,\forall i\in[m],\, (d^i)^\top Z d^j \ge 0,\,\forall i,j\in [m]}    
    \end{equation}
with $\set{Z\in\S^{n+2}:\, Z_{n+2,n+2}=1}$.   
The way this formulation was derived in \cite{burer2024slightly} is by introducing a new variable $t$ modeling $\|x\|_2^2$ and the homogenization variable $x_0$, and then applying linear-RLT inequalities (i.e., $(d^i)^\top Z d^j \ge 0,\,\forall i,j\in [m]$) and SOC-RLT inequalities (i.e., $PZd^i\in\L^{n+2},\,\forall i\in[m]$) to the set $\widetilde{\cD}:=\cone\{ww^\top:~ w\in \widetilde{S}\}$, where
\[\
\widetilde{S}:=\set{w=(x,t,x_0)\in\R^{n+2}:~ x_0=1, ~\|x\|_2^2=tx_0,~\ip{d^i,w
}\ge 0,\, \forall i\in[m]}.
\]
Clearly, $S$ is the projection of $\widetilde{S}$ onto the first $n$ coordinates, i.e., $S=\Proj_{x}(\widetilde{S})$. Moreover, as 
\[
\conv(S)=\conv\left(\Proj_{x}(\widetilde{S})\right)=\Proj_{x}\left(\conv(\widetilde{S})\right),
\]
(where $\conv(S)$ stands for the convex hull of set $S$) 
we may thus focus on obtaining $\conv(\widetilde{S})$ or a relaxation for it.

\section{Domination of Kronecker RLT inequalities
}

In order to project $\cC_m$ back to the space of $\cD$, we define the following linear maps of $Z\in\S^{n+2}$
\[
\pi_X(Z):=\begin{bmatrix}
    Z_{11} & \ldots & Z_{1n}\\
    \vdots & \ddots & \vdots\\
    Z_{n1} & \ldots & Z_{nn}
\end{bmatrix} \in \S^n,
\quad \pi_x(Z):=\begin{bmatrix}
    Z_{1,n+2}\\ \vdots \\ Z_{n,n+2}
\end{bmatrix}\in\R^n,
\]
and denote $\pi(Z):=(\pi_X(Z),\pi_x(Z))\in \S^n \oplus \R^n$. 

Although $\cD \subsetneq \pi(\cC_m)$ for any $m\ge 2$,  \citet{burer2024slightly} shows that for the case $m=2$ the relaxation (\ref{eqn:burer}) can be slightly modified to become exact, i.e., setting $\ip{Q,Z}\ge 0$ and $(d^1)^\top Zd^2=0$ and keeping everything else the same.
Indeed, when $m=2$ this modification is justified and leads to a valid convex relaxation that is shown to be exact in \cite{burer2024slightly}.
However, this modification does not generalize to $m\ge 3$. 

For $m\ge 3$, \citet{burer2024slightly} numerically compares the strength of (\ref{eqn:burer}) against a lifted SDP formulation strengthened by the Kronecker RLT inequalities proposed in \cite{anstreicher2017kronecker}. The idea for the Kronecker inequalities is as follows: each ball constraint $\|x-c^i\|_2 \le r_i$ is equivalent to
\[
\cA(x-c^i,r_i):=\begin{bmatrix}
    r_i & & & x_1-c^i_1\\
    & \ddots & & \vdots \\
    & & r_i & x_n-c^i_n \\
    x_1-c^i_1 & \ldots & x_n-c^i_n & r_i
\end{bmatrix}\succeq 0.
\]
Recall also that for any two positive semidefinite matrices $A^i\in\S^{n_1}$ and $A^j\in\S^{n_2}$, we always have their Kronecker product $A^i\otimes A^j$ is positive semidefinite as well.
Thus, if a point $x\in\R^n$ satisfies two ball constraints $\|x-c^i\|_2 \le r_i$ and $\|x-c^j\|_2 \le r_j$, then the Kronecker product $\cA(x-c^i,r_i)\otimes \cA(x-c^j,r_j)$ is positive semidefinite as well.
Note that the entries in $\cA(x-c^i,r_i)\otimes \cA(x-c^j,r_j)$ are quadratic functions of $x$ and as such they can be expressed as linear functions of the rank-one matrix $\begin{pmatrix}x\\1\end{pmatrix}\begin{pmatrix}x^\top&1\end{pmatrix}$. In order to obtain the Kronecker inequalities, we simply replace the quadratic terms $x_kx_l$ in $\cA(x-c^i,r_i)\otimes \cA(x-c^j,r_j)$ by $X_{kl}$ and $x_k^2$ by $X_{kk}$. Let $\cK_{ij}$ denote the resulting linear map from $\S^n \oplus \R^n$ to $\S^{(n+1)^2}$ of the matrix variable $X$ and the vector $x$ (see \cite{anstreicher2017kronecker} for formal entry-wise definition of $\cK_{ij}$). Hence, we arrive at the conclusion that any $(X,x)\in\cD$ must satisfy
 \begin{equation} \label{eqn:kron}
    \cK_{ij}(X,x)\succeq 0, \quad \forall i,j\in[m] .
\end{equation}

The numerical study in \cite{burer2024slightly} indicates that the lifted relaxation (\ref{eqn:burer}) is stronger than the Kronecker inequalities (\ref{eqn:kron}), and as a result \cite{burer2024slightly} conjectures that this must be true theoretically as well. In this note we resolve this conjecture and show that this is indeed the case. 

\begin{theorem}\label{thm:main}
For any $m\ge2$ and for all $Z\in \cC_m$, we have     $\cK_{ij}(\pi_X(Z),\pi_x(Z))\succeq 0$ for all $i,j\in[m]$. 
\end{theorem}

Indeed, we prove a slightly stronger form of Theorem~\ref{thm:main}.  
Note that for all $i,j\in[m]$ where $i\neq j$, by defining
\[
\cC_{\{i,j\}} :=\set{Z\in\S^{n+2}_+:\, \ip{Q,Z} = 0,
    \, PZd^i\in\L^{n+2},~ PZd^j\in\L^{n+2},~ (d^i)^\top Z d^j \ge 0},    
\]
we observe that $\cC_m=\bigcap\limits_{i\in[m]}\bigcap\limits_{j<i} \cC_{\{i,j\}}$. In the view of this, in fact, we will prove the following stronger claim:
\begin{theorem}\label{thm:main_refined}
For any $m\ge2$ and $i,j\in[m]$ where $i\neq j$ we have the relation
\begin{align*}
Z\in \cC_{\{i,j\}} \implies \cK_{ij}(\pi_X(Z),\pi_x(Z))\succeq 0.
\end{align*}
\end{theorem}
Clearly, Theorem~\ref{thm:main} is a corollary of Theorem~\ref{thm:main_refined}. Thus, from now on we focus on the case of  $m=2$, $(i,j)=(1,2)$, and consider the set     
\begin{equation}\label{eqn:burer_2}
    \cC_2=\set{Z\succeq 0:~ \ip{Q,Z} = 0,
    ~PZd^i\in\L^{n+2},\,\forall i\in [2] ,~ (d^1)^\top Z d^2 \ge 0}.    
    \end{equation}

We start with some preliminaries that relate to properties of extreme rays of $\cC_2$ that are spanned by rank-one matrices in Section~\ref{sec:Kron:preliminary}.
The key ingredient of proving Theorem \ref{thm:main_refined} is a decomposition of non-rank-one extreme rays of $\cC_2$, which we present in Section~\ref{sec:Kron:decomposition}. We believe this result is of separate interest in order to further understand the geometry of relaxation \eqref{eqn:burer}. 
Finally, in Section~\ref{sec:Kron:proof} we give the proof of Theorem~\ref{thm:main_refined}.

\subsection{Preliminaries}\label{sec:Kron:preliminary}

We start with the following simple fact on the relationship between matrices $P,Q$ and the second-order cone. Given a set $\cX$, we let $\bd(\cX)$ denote its boundary.
\begin{fact}\label{fact:Q-P_relation}
    For any $w\in\R^{n+2}$, the relation $w^{\top}Qw\ge 0$ holds if and only if $Pw\in\L^{n+2}\cup(-\L^{n+2})$. Moreover,  $w^{\top}Qw= 0$ holds if and only if $Pw\in \left(\bd(\L^{n+2})\cup\bd(-\L^{n+2})\right)$. 
\end{fact}
\begin{proof}
    Write $w^\top=(x^\top,t,x_0)$. Then,  $w^{\top}Qw\ge 0$ if and only if $tx_0\ge \|x\|_2^2$ which is equivalent to $4\|x\|_2^2+(t-x_0^2)\le (t+x_0)^2$.
    Thus, $w^{\top}Qw\ge 0$ if and only if     
    \[
    Pw=\begin{bmatrix}2x\\t-x_0\\t+x_0\end{bmatrix}\in \L^{n+2}\cup(-\L^{n+2}). 
    \]
    Similarly, $w^{\top}Qw=0$ holds if and only if $4\|x\|_2^2+(t-x_0^2)= (t+x_0)^2$, which holds if and only if $Pw \in \left(\bd(\L^{n+2})\cup\bd(-\L^{n+2})\right)$.  
\end{proof}

We characterize rank-one extreme rays of $\cC_m$ (recall from (\ref{eqn:burer})) for any $m\ge 2$. 

\begin{proposition}\label{prop:rank_one_rays}
    Let $m\ge 2$ and $w^{\top}:=(x^{\top},t,x_0)$ be such that $ww^{\top}$ spans an extreme ray of $\cC_m$. Then, $\|x\|_2^2=tx_0$ and $x_0\ne 0$. Moreover, if $x_0>0$, then $w^{\top}d^i\ge 0,\forall i\in [m]$. 
\end{proposition}
\begin{proof}
As $Z=ww^\top$ spans an extreme ray of $\cC_m$, $w\neq0$. 
    $\ip{Q,ww^{\top}}=0$ implies $\|x\|_2^2=tx_0$. Assume for contradiction that $x_0=0$. Then, $x=0$ and $w^{\top}=\alpha (0^{\top},1,0)$ for some $\alpha\in \R$. Since $d^i_{n+1}=-1$ for all $i\in[m]$, we have $(Pww^{\top}d^i)^{\top}=-\alpha^2 (0^{\top},1,1)$. Thus, the constraint $Pww^{\top}d^i\in\L^{n+2}$ implies $\alpha=0$, and so $w=0$ which is a contradiction. Hence, we conclude $x_0\ne 0$.

    Since $\ip{Q,ww^{\top}}=0$, from Fact \ref{fact:Q-P_relation} we deduce $Pw\in \L^{n+2}\cap (-\L^{n+2})$. After choosing $x_0>0$ we have $Pw\in \L^{n+2}$. Since $P$ is invertible and $w\ne 0$ we have $Pw\ne 0$. Hence, the constraint $Pw(w^{\top}d^i)=PZd^i\in \L^{n+2}$ implies $w^{\top}d^i\ge 0$ for all $i\in [m]$. 
\end{proof}

In particular, Proposition~\ref{prop:rank_one_rays} states that if $\Zb=ww^\top$ (i.e., $\rank(\Zb)=1$) spans an extreme ray of $\cC_m$, then by properly scaling $w$ so that $w_{n+2}=1$ we observe that  $w\in \widetilde{S}$. 
Then, by the origin of the Kronecker inequalities we deduce that $\cK_{ij}(\pi_X(\Zb),\pi_x(\Zb))\succeq 0$ holds for all $i,j\in[m]$.

\subsection{Decomposition of non-rank-one extreme rays of $\cC_2$}
\label{sec:Kron:decomposition}

\citet{burer2024slightly} considers the following slightly modified version of $\cC_2$: 
    \begin{equation}\label{eqn:burer_exact}
    \tilde{\cC}_2:=\set{Z\succeq 0:~ \ip{Q,Z} \ge 0,
    ~PZd^i\in\L^{n+2},\,\forall i\in [2] ,~ (d^1)^\top Z d^2 = 0}.    
    \end{equation}

Recall that a closed convex cone contained inside a PSD cone is called \emph{rank one generated} (ROG) if all of its extreme rays are generated by rank one matrices. See \cite{hildebrand2016spectrahedral} and \cite{argue2020necessary} for properties of ROG cones. 

We restate the following results from \cite{burer2024slightly} on the following cones being ROG. 

\begin{theorem}[{restatement of \cite[Theorem 1]{burer2024slightly}}]\label{thm:ROG_twoSOC}
    $\tilde{\cC}_2$ is ROG. 
\end{theorem}

\begin{proposition}[{restatement of \cite[Lemma 
3]{burer2024slightly}}]\label{prop:ROG_oneSOC}
    For any $d\in\R^{n+2}$, the set 
    \[
    \cC_1:=\set{Z\succeq 0:~ \ip{Q,Z} \ge 0,~PZd\in\L^{n+2}}
    \]
    is ROG. 
\end{proposition}

We also need the following straightforward lemma on the faces of ROG cones. 

    \begin{lemma}[part of Lemma 3 in \cite{argue2020necessary}]\label{lem:ROG_face}
        Let $K$ be an ROG cone, and $\ip{\alpha,\xi}\le \beta$ be a valid inequality for $K$. Then, $K\cap \{\xi:~ \ip{\alpha,\xi}=\beta\}$ is ROG. 
    \end{lemma}
    \begin{proof}
        Let $\xi^*$ span an extreme ray of $K\cap \{\xi:~ \ip{\alpha,\xi}=\beta\}$. Since $\xi^*\in K$ and $K$ is ROG, $\xi^*$ admits a representation in terms of conic combination of rank one matrices in $K$, i.e., $\xi^* = \sum_{\ell\in[r]} w^\ell (w^\ell)^\top$ for some $r$ and $w^\ell (w^\ell)^\top \in K$. Since $\ip{\alpha,\xi}\le \beta$ is a valid inequality for $K$ and $\xi^*\in K\cap \{\xi:~ \ip{\alpha,\xi}=\beta\}$, all these rank-one matrices $w^\ell (w^\ell)^\top$ for all $\ell\in[r]$ must also satisfy $\ip{\alpha,\xi}\le \beta$ at equality. Hence, $K\cap \{\xi:~ \ip{\alpha,\xi}=\beta\}$ is ROG. 
    \end{proof}

Although $\cC_2$ is not ROG, we have the following characterization of its non-rank-one extreme rays. 
    \begin{theorem}\label{thm:non_rank_one_extray}
        Let $\Zb$ span an extreme ray of $\cC_2$. If $r=\rank(\Zb)>1$, then there exists $x^1,\ldots,x^r\in\R^n,\alpha_1,\ldots,\alpha_r>0$, such that 
        \[
        \Zb=\sum_{i=1}^r \alpha_i\begin{bmatrix}
            x^i\\ \|x^i\|_2^2 \\1
        \end{bmatrix}\begin{bmatrix}
            x^i\\ \|x^i\|_2^2 \\1
        \end{bmatrix}^{\top},
        \]
        where $\|x^1-c^1\|_2<r_1,~\|x^1-c^2\|_2<r_2$, $\|x^i-c^1\|_2=r_1$ 
        for all $2\le i\le r$, and $P\Zb d^2\in \bd (\L^{n+2})$.  
    \end{theorem}

    \begin{proof}
        Based on Theorem \ref{thm:ROG_twoSOC} and Lemma \ref{lem:ROG_face},  the set 
        \[
        \set{Z\succeq 0:~ \begin{array}{l}
        \ip{Q,Z} = 0,\\
        PZd^i\in\L^{n+2},\,\forall i\in[2],\\
        (d^1)^\top Z d^2 = 0
        \end{array} 
         }
        \]
        is ROG. Since $\Zb$ has rank $r>1$ we must have $(d^1)^{\top}\Zb d^2>0$. 
        In addition, as $\ip{Q,\Zb} = 0$, by Lemma \ref{lem:ROG_face}  we also have $P\Zb d^i\in \bd(\L^{n+2})$ for $i\in [2]$.
        
        Note that by Proposition \ref{prop:ROG_oneSOC} for each $i\in [2]$ the sets 
        \[
        \set{Z\succeq 0:~ \ip{Q,Z} = 0,~PZd^i\in\L^{n+2}}
        \]
         are ROG. 
         Thus, we can write $\Zb=\sum_{i=1}^r \alpha_i w^i (w^i)^{\top}$, where for each $i\in[r]$, $\alpha_i>0$ and $w^i$ satisfies $w^i\neq0$, $\ip{Q,w^i(w^i)^\top}=0,~P(w^i(w^i)^\top)d^1\in \L^{n+2}$. 
        For any $i\in[r]$, let $(w^i)^{\top}=((x^i)^{\top},t_i,y_i)$. Replacing $w^i$ by $-w^i$ if necessary, we may assume $y_i\ge 0$. Also, the constraint $\ip{Q,w^i(w^i)^{\top}}=0$ implies $t_iy_i=\|x^i\|_2^2$. We argue that $y_i\ne 0$, since otherwise we have $x^i=0$, $(w^i)^{\top}=(0,\beta,0)$ for some $\beta \ne 0$, and then as $d^i_{n+1}=-1$ we have $Pw^i(w^i)^{\top}d^1=-\beta^2\begin{bmatrix}
            0\\1\\1
        \end{bmatrix}$ which violates $Pw^i(w^i)^{\top}d^1\in\L^{n+2}$. Thus, $y_i> 0$. Upon rescaling $\alpha_i$ and $w^i$ we may assume $y_i=1$ and hence $t_i=\|x^i\|_2^2$. Thus, we deduce that each $i\in[r]$ the vector $w^i$ is of the form  
        \[
        w^i=\begin{bmatrix}
            x^i\\ \|x^i\|_2^2 \\1
        \end{bmatrix},~\text{ and }~  Pw^i\in \L^{n+2}.
        \]
        As $Pw^i ((w^i)^\top d^1)\in\L^{n+2}$ and $\L^{n+2}$ is a pointed convex cone, we also get $\ip{d^1,w^i}\ge 0$.
        
        Now note that $P\Zb d^1 = \sum_{i=1}^r \alpha_i \ip{d^1,w^i} Pw^i$, while we also know  $P\Zb d^1\in \bd(\L^{n+2})$ and $Pw^i\in \L^{n+2}$ and $\ip{d^1,w^i}\ge 0$ for all $ i\in [r]$. Then, since $\L^{n+2}$ is a strictly convex cone (i.e., the boundary does not contain any polyhedron of dimension at least two) and $\alpha_i>0$ for all $i\in[r]$, we deduce that exactly one of the terms $\ip{d^1,w^i}$ is positive and the rest are zero. Without loss of generality by rearranging the indices of $w^i$'s, we conclude that $\ip{d^1,w^1}>0$ and $\ip{d^1,w^i}=0$ for  all $2\le i\le r$. This implies $\|x^1-c^1\|_2<r_1$ and $\|x^i-c^1\|_2=r_1$ for all $2\le i\le r$. Then,
        \[
        0<(d^1)^{
        \top}\Zb d^2=\sum_{i=1}^r \ip{d^1,w^i} \ip{w^i, d^2}=\ip{d^1,w^1} \ip{w^1, d^2}.
        \]
        Hence, $\ip{w^1, d^2}>0$, i.e., $\|x^1-c^2\|_2<r_2$.
    \end{proof}

\subsection{Proof of Theorem \ref{thm:main_refined}}\label{sec:Kron:proof}

Recall the linear map $\cK_{ij}$ from (\ref{eqn:kron}). Since we are considering the case $m=2$, we denote the map as $\cK$. 

Let us now define $\cK$ more formally for our proof. For convenience, we denote the arrow matrices as $\cA_i(x):=\cA(x-c^i,r_i),~ i\in [2]$. Let $X_j$ denote the $j$-th column of $X$, and let $H_j(X,x)$ be the matrix obtained from $(x_j-c^2_j)\cA_1(x)$ by replacing $x_jx$ terms by $X_j$, i.e., 
    \[
    H_j(X,x):=\begin{bmatrix}
        (x_j-c^2_j)r_1 I & X_j-x_jc^1-c^2_jx+c^2_jc^1\\(X_j-x_jc^1-c^2_jx+c^2_jc^1)^{\top} & (x_j-c^2_j)r_1
    \end{bmatrix}.
    \]

    Thus, the Kronecker RLT constraint derived from the two ball constraints $\|x-c^i\|_2\le r_i,i\in [2]$ can be written as
    \[
    \cK(X,x)=\begin{bmatrix}
        r_2\cA_1(x) & & & H_1(X,x)\\
        & \ddots & & \vdots\\
        & & r_2\cA_1(x) & H_n(X,x)\\
        H_1(X,x) & \ldots & H_n(X,x) & r_2\cA_1(x)
    \end{bmatrix}\succeq 0.
    \]

Note that both $\cA_i(x)$ and $H_j(X,x)$ are affine maps of $(X,x)$, and thus $\cK(X,x)$ is an affine map of $(X,x)$ as well and so we arrive at the following result. 
    \begin{lemma}\label{lem:kron_equivalence}
        For any $x_1,\ldots,x_r\in\R^n$ and $\alpha_1,\ldots,\alpha_r\in \R$, we have
\[
        \cK \left(\sum_{i=1}^r \alpha_ix_ix_i^{\top},\sum_{i=1}^r \alpha_ix_i \right)=\sum_{i=1}^r \alpha_i \left(\cA_2(x)\otimes \cA_1(x) \right). 
\]
    \end{lemma}

    Based on the characterization of rank-one and non-rank-one extreme rays of $\cC_2$ given in Proposition \ref{prop:rank_one_rays} and Theorem \ref{thm:non_rank_one_extray}, we conclude that to finish the proof of Theorem~\ref{thm:main_refined}, it suffices to prove the following result.
    \begin{theorem}\label{thm:better_than_kron}
        Let $\Zb$ span an extreme ray of $\cC_2$ and suppose $r:=\rank(\Zb)>1$. Define $\bar{X}:=\pi_X(\Zb),\bar{x}:=\pi_x(\Zb)$. Then, $\cK(\bar{X},\bar{x})\succeq 0$. 
    \end{theorem}

    The proof of Theorem~\ref{thm:better_than_kron} is mostly algebraic. We will use the following lemma to simplify our computations. 
    \begin{lemma}\label{lem:compare_kron}
        For any $\cV\in\R^{(n+1)^2}$ we use the notation
        \[\cV^{\top}:=\left(
            (v^1)^{\top},  b_1,  \ldots , (v^{n+1})^{\top} , b_{n+1}
        \right),\]
        where $v^j\in\R^n,~b_j\in \R$ for all $j\in [n+1]$. 
        For any $x\in \R^n$ such that either $\|x-c^1\|_2=r_1$ or $\|x-c^1\|_2<r_1$ and $\|x-c^2\|_2<r_2$, we have 
        \[
        \cV^{\top}\left(\cA_2(x)\otimes \cA_1(x)\right)\cV 
        \ge \frac{1}{r_1r_2} \left(r_2^2-\|x-c^2\|_2^2 \right) \left\|r_1v^{n+1}+b_{n+1}(x-c^1)\right\|_2^2.
    \]
    \end{lemma}

    We defer the proof of this lemma to appendix.

        \begin{proof}[Proof of Theorem \ref{thm:better_than_kron}]
        Based on Theorem~\ref{thm:non_rank_one_extray}, we can write 
       \[
        \Zb=\sum_{i=1}^r \alpha_i\begin{bmatrix}
            x^i\\ \|x^i\|_2^2 \\1
        \end{bmatrix}\begin{bmatrix}
            x^i\\ \|x^i\|_2^2 \\1
        \end{bmatrix}^{\top},
        \]
        where $\alpha_i>0$ for all $i\in[r]$, $\|x^1-c^1\|_2<r_1,~\|x^1-c^2\|_2<r_2$, $\|x^i-c^1\|_2=r_1$ for all $2\le i\le r$, and $P\Zb d^2\in \bd (\L^{n+2})$.  
        
        In order to complete the proof, we will show that $\cV^\top \cK(\bar{X},\bar{x})\cV\ge 0$ holds for any $\cV\in\R^{(n+1)^2}$. Let us write $\cV\in\R^{(n+1)^2}$ as $\cV^{\top}=\left(
            (v^1)^{\top},  b_1,  \ldots , (v^{n+1})^{\top} , b_{n+1}
        \right)$, where $v^j\in\R^n,b_j\in \R$ for all $j\in[n+1]$. 
        Then, by Lemmas~\ref{lem:kron_equivalence}~and~\ref{lem:compare_kron}, it suffices to prove that for any $i\in[r]$ we have
        \[
            \frac{1}{r_1r_2}\sum_{i=1}^r \alpha_i  \left(r_2^2-\|x^i-c^2\|_2^2 \right) \left\|r_1v^{n+1}+b_{n+1}(x^i-c^1)\right\|_2^2    \ge 0. 
        \]
        Let us define $\eta_i:= r_2^2-\|x^i-c^2\|_2^2$ for all $i\in[r]$. Thus, we need to show
        \begin{equation}\label{eq:desired}
            \frac{1}{r_1r_2}\sum_{i=1}^r \alpha_i  \eta_i \left\|r_1v^{n+1}+b_{n+1}(x^i-c^1)\right\|_2^2    \ge 0. 
        \end{equation}
        
        Note that for all $i\in[r]$ we have
           \begin{align*}
               P\begin{bmatrix}
            x^i\\ \|x^i\|_2^2 \\1
        \end{bmatrix}\begin{bmatrix}
            x^i\\ \|x^i\|_2^2 \\1
            \end{bmatrix}^{\top} d^2
            & = \begin{bmatrix}2I & 0 & 0\\ 0 & 1 & -1\\0 & 1 & 1
\end{bmatrix} 
             \begin{bmatrix}
            x^i\\ \|x^i\|_2^2 \\1
        \end{bmatrix}\begin{bmatrix}
            x^i\\ \|x^i\|_2^2 \\1
            \end{bmatrix}^{\top}
            \begin{bmatrix}2c^2\\-1\\r_2^2-\|c^2\|_2^2\end{bmatrix} \\
            &=\begin{bmatrix}
            2x^i\\ \|x^i\|_2^2 -1 \\ \|x^i\|_2^2 +1
            \end{bmatrix}
            \left(2(c^2)^\top x^i- \|x^i\|_2^2 + r_2^2-\|c^2\|_2^2\right) \\
            &=\left(r_2^2-\|x^i-c^2\|_2^2\right) 
            \begin{bmatrix}
            2x^i\\ \|x^i\|_2^2 -1 \\ \|x^i\|_2^2 +1
            \end{bmatrix}
            = \eta_i \begin{bmatrix}
            2x^i\\ \|x^i\|_2^2 -1 \\ \|x^i\|_2^2 +1
            \end{bmatrix}.
            \end{align*}
            Hence, 
            \begin{align*} 
                P\Zb d^2
             &=\sum_{i=1}^r \alpha_i \eta_i \begin{bmatrix}
            2x^i\\ \|x^i\|_2^2 -1 \\ \|x^i\|_2^2 +1
            \end{bmatrix} 
                = \begin{bmatrix}
            2 \sum_{i=1}^r \alpha_i \eta_i x^i\\ 
            \left(\sum_{i=1}^r \alpha_i \eta_i  \|x^i\|_2^2 \right) - \left(\sum_{i=1}^r \alpha_i \eta_i \right) \\ 
             \left(\sum_{i=1}^r \alpha_i \eta_i  \|x^i\|_2^2 \right) + \left(\sum_{i=1}^r \alpha_i \eta_i \right) 
            \end{bmatrix}.
            \end{align*}
            Thus, using $P\Zb d^2\in \L^{n+2}$ we deduce that 
            \begin{align}
            & \sum_{i=1}^r \alpha_i \eta_i  + \sum_{i=1}^r \alpha_i \eta_i \|x^i\|_2^2 = \sum_{i=1}^r \left(\alpha_i \eta_i  \left( \|x^i\|_2^2 + 1 \right) \right) \ge 0 \label{eq:PZbd^2-equiv1}, \\
        \text{and }~~&\left(\sum_{i=1}^r \alpha_i \eta_i \right)\left(\sum_{i=1}^r \alpha_i \eta_i \|x^i\|_2^2 \right) \ge \left\|\sum_{i=1}^r \alpha_i \eta_i x^i\right\|_2^2. \label{eq:PZbd^2-equiv2}
        	    \end{align}
	    Therefore, we must have $\sum_{i=1}^r \alpha_i \eta_i \ge0$  and $\sum_{i=1}^r \alpha_i \eta_i \|x^i\|_2^2 \ge0$  
	since by \eqref{eq:PZbd^2-equiv1} the sum of these terms is nonnegative and by \eqref{eq:PZbd^2-equiv2} their product is nonnegative.
    
	We define and examine the following $(n+1) \times (n+1)$ matrix
        \[
        M:=\begin{bmatrix}
            \left(\sum_{i=1}^r \alpha_i \eta_i \right)I & \sum_{i=1}^r \alpha_i \eta_i x^i  \\
            \sum_{i=1}^r \alpha_i \eta_i (x^i)^{\top} &  \sum_{i=1}^r \alpha_i \eta_i \|x^i\|_2^2
        \end{bmatrix} .
        \]
        We claim that $M\succeq0$ as a result of \eqref{eq:PZbd^2-equiv1} and \eqref{eq:PZbd^2-equiv2}. First, consider the case when at least one of the terms $\sum_{i=1}^r \alpha_i \eta_i$ and $\sum_{i=1}^r \alpha_i \eta_i \|x^i\|_2^2$ is equal to zero. Then, by \eqref{eq:PZbd^2-equiv2}, we deduce that $\sum_{i=1}^r \alpha_i \eta_i x^i = 0$ and from \eqref{eq:PZbd^2-equiv1} we conclude that the other term must be nonnegative and so $M\succeq0$ holds in this case.
        Now, consider the case when both $\sum_{i=1}^r \alpha_i \eta_i$ and $\sum_{i=1}^r \alpha_i \eta_i \|x^i\|_2^2$ are nonzero. Since both terms are always nonnegative, in this case they both must  be positive. Then, both of the diagonal block matrices in $M$ are positive definite, and by Schur Complement Lemma $M\succeq0$ if and only if 
        \begin{align*} 
        &\sum_{i=1}^r \alpha_i \eta_i \|x^i\|_2^2 -  \left(\sum_{i=1}^r \alpha_i \eta_i \right)^{-1} \underbrace{\left( \sum_{i=1}^r \alpha_i \eta_i x^i \right)^{\top}  I \left( \sum_{i=1}^r \alpha_i \eta_i x^i \right)}_{=\norm{\sum_{i=1}^r \alpha_i \eta_i x^i}_2^2} \ge 0 .
        \end{align*}
        By recalling that $\sum_{i=1}^r \alpha_i \eta_i>0$ and rearranging the above inequality, we observe that it is precisely the same as \eqref{eq:PZbd^2-equiv2}. Thus, we conclude that $M\succeq0$ holds in this case as well.
	    
	Now, as $M\succeq0$, for any $v\in\R^n$ we have
	\begin{align} \label{eq:implied}
	0 \le \begin{bmatrix} v\\1\end{bmatrix}^{\top}M\begin{bmatrix} v\\1\end{bmatrix}
	= \sum_{i=1}^r \alpha_i \eta_i  \|v+x^i\|_2^2.
	\end{align}	
	Finally, note that whenever $b_{n+1}\ne 0$ we have $\|r_1v^{n+1}+b_{n+1}(x^i-c^1)\|_2^2=b_{n+1}^2\|\frac{r_1v^{n+1}}{b_{n+1}}-c^1+x^i\|_2^2$, hence \eqref{eq:implied} implies the desired relation of \eqref{eq:desired}. The case when $b_{n+1}=0$ follows from continuity letting $b_{n+1}\to 0$. This completes the proof. 
       \end{proof}

\subsection{Domination of Zhen et al RLT inequalities}

In this section we examine an inequality proposed by Zhen et al. in \cite{zhen2021extension} to strengthen the Shor SDP relaxation of the sets of form $S$. Once again, in the numerical study of \cite{burer2024slightly}, it was identified that this inequality is redundant for $\cC_m$, and it was conjectured in \cite{burer2024slightly} that this must be theoretically true as well. 

The main idea of Zhen et al.\ \cite{zhen2021extension}'s inequality is as follows. 
Recall that given two vectors $u\in\R^k$ and $v\in\R^\ell$, we have $\|uv^\top\|_{m2}=\|u\|_2\|v\|_2$, where $\|\cdot\|_{m2}$ denotes the matrix 2-norm. 
Then, given any two SOC constraints $\|x-c^i\|_2\le r_i$ and $\|x-c^j\|_2\le r_j$ with $i\neq j$, we have
\begin{align*}
\|xx^\top-x(c^j)^\top - c^ix^\top +c^i(c^j)^\top\|_{m2} &= \|(x-c^i)(x-c^j)^\top\|_{m2} \\
&= \|x-c^i\|_2 \|x-c^j\|_2 \leq r_i r_j. 
\end{align*}
Now, by defining a matrix variable $X$ to capture $xx^\top$, this inequality can be linearized in the lifted space as 
\begin{align}
&\|X-x(c^j)^\top-c^ix^\top + c^i(c^j)^\top\|_{m2} \leq r_i r_j \label{eq:Zhen} \\
&\iff  
\underbrace{\begin{bmatrix}
r_i^2 I_n & X-x(c^j)^\top-c^ix^\top + c^i(c^j)^\top \\
X-c^jx^\top-x(c^i)^\top + c^j(c^i)^\top & r_j^2 I_n
\end{bmatrix}}_{:=\cZ_{i,j}(X,x)}
\succeq0 . \notag
\end{align}
Note that this inequality~\eqref{eq:Zhen}, i.e., $\cZ_{i,j}(X,x)\succeq0$, is linear in terms of the vector variable $x$ and the matrix variable $X$. Thus, Zhen et al. suggest to add the inequality $\cZ_{i,j}(X,x)\succeq0$ for all $i,j\in[m]$ to strengthen the Shor SDP relaxation of the set $S$.

We next show that the inequality~\eqref{eq:Zhen} follows from adding linear- and SOC-RLT inequalities to the standard Shor relaxation. Thus, the inequalities~\eqref{eq:Zhen} are redundant for the lifted relaxation  $\cC_m$. 
\begin{theorem}\label{thm:Zhen}
For any $m\ge2$ and $i,j\in[m]$ where $i\neq j$ we have the relation
\begin{align*}
\set{Z\in \cC_{\{i,j\}},~ Z_{n+2,n+2}=1 } \implies \cZ_{ij}(\pi_X(Z),\pi_x(Z))\succeq 0.
\end{align*}
\end{theorem}
\begin{proof}
Consider any $Z\in \cC_{\{i,j\}}$ satisfying $Z_{n+2,n+2}=1$. 
To simplify our notation, let us set $X:=\pi_X(Z)$ and $x:=\pi_x(Z)$.
Thus, by the definition of $Z$, observe that
\[
Z=\begin{bmatrix}
X & Z_{t,x} & x \\
Z_{t,x}^\top & Z_{t,t}  & t \\
x^\top & t & 1
\end{bmatrix}
\]
Now since $Z\succeq0$, we deduce $\begin{bmatrix}
X &  x \\
x^\top & 1
\end{bmatrix}\succeq0$, which by Schur Complement Lemma implies that $X-xx^\top \succeq0$.

Using the definitions of $Q$ and $Z$, we observe that 
\[
QZ = \begin{bmatrix}-2I & 0 & 0\\0 & 0 & 1\\0 & 1 & 0\end{bmatrix} 
\begin{bmatrix}
X & Z_{t,x} & x \\
Z_{t,x}^\top & Z_{t,t}  & t \\
x^\top & t & 1
\end{bmatrix}
= \begin{bmatrix}
-2X & -2Z_{t,x} & -2x \\
x^\top & t  & 1 \\
Z_{t,x}^\top & Z_{t,t}  & t
\end{bmatrix}.
\]
Then, as $\ip{Q,Z}=\tr(QZ)=0$, we deduce that $0=-2\tr(X)+t+t$, i.e., $t=\tr(X)$.

Moreover, by definition of $P$ and $d^i$, we have
\begin{align*}
PZd^i &= \begin{bmatrix}2I & 0 & 0\\ 0 & 1 & -1\\0 & 1 & 1 \end{bmatrix}
\begin{bmatrix}
X & Z_{t,x} & x \\
Z_{t,x}^\top & Z_{t,t}  & t \\
x^\top & t & 1
\end{bmatrix} d^i \\
&= \begin{bmatrix}
2X & 2Z_{t,x} & 2x \\
Z_{t,x}^\top - x^\top & Z_{t,t}-t  & t-1 \\
Z_{t,x}^\top + x^\top & Z_{t,t}+t  & t+1
\end{bmatrix} 
\begin{bmatrix}2c^i\\-1\\r_i^2-\|c^i\|_2^2\end{bmatrix} \\
& = \begin{bmatrix}4Xc^i - 2Z_{t,x} + 2(r_i^2-\|c^i\|_2^2) x \\
2 (Z_{t,x} - x)^\top c^i - (Z_{t,t}-t ) + (t-1)(r_i^2-\|c^i\|_2^2) \\
2 (Z_{t,x} + x)^\top c^i - (Z_{t,t}+t ) + (t+1)(r_i^2-\|c^i\|_2^2)
\end{bmatrix} \\
& = \begin{bmatrix} 4Xc^i - 4xx^\top c^i  - 2Z_{t,x} + 2tx + 2\delta_i x \\
2 Z_{t,x}^\top c^i - Z_{t,t} + t(r_i^2-\|c^i\|_2^2) - \delta_i \\
2 Z_{t,x}^\top c^i - Z_{t,t} + t(r_i^2-\|c^i\|_2^2) +\delta_i
\end{bmatrix}
= \begin{bmatrix} \xi^i + 2 \delta_i x \\
\beta_i - \delta_i \\
\beta_i + \delta_i
\end{bmatrix}
\end{align*}
where in the last two equations we used the definitions $\delta_i := 2 x^\top c^i - t + r_i^2-\|c^i\|_2^2$, $\beta_i :=2 Z_{t,x}^\top c^i - Z_{t,t} + t(r_i^2-\|c^i\|_2^2) $ and $\xi^i := 4Xc^i - 4xx^\top c^i  - 2Z_{t,x} + 2tx $.
Since $PZd^i \in\L^{n+2}$, we must thus have $\beta_i + \delta_i \ge0$ and 
\[
\|\xi^i + 2 \delta_i x \|_2^2 + (\beta_i-\delta_i)^2 \le (\beta_i+\delta_i)^2  \quad \iff\quad
\|\xi^i + 2 \delta_i x\|_2^2 \le 4\beta_i \delta_i .
\]
Because we have both $\beta_i + \delta_i \ge0$ and $\beta_i  \delta_i \ge0$, we conclude that both $\beta_i$ and $\delta_i$ must be nonnegative.

Note that $\delta_i\ge0$ is equivalent to
\[
0 \leq \delta_i = - t + 2 x^\top c^i  + r_i^2-\|c^i\|_2^2 
= - \tr(X) + 2 x^\top c^i  + r_i^2-\|c^i\|_2^2.
\]
Thus, we arrive at
\begin{equation}\label{eq:impliedLMI_s}
    \tr(X-x(c^i)^\top-c^ix^\top + c^i(c^i)^\top) \leq r_i^2.
\end{equation}
In fact, \eqref{eq:impliedLMI_s} can be obtained in much simpler way by noting that squaring both sides of $\|x-c_i\|_2\le r_i$ results in $(x-c^i)^\top(x-c^i)=x^\top x - (c^i)^\top x - x^\top c^i + (c^i)^\top c^i\leq r_i^2$ and then writing the nonlinear term $x^\top x$ as $\tr(X)$ in the lifted space. 
Typically \eqref{eq:impliedLMI_s} is added to the Shor SDP relaxation as the most basic constraint to capture the underlying quadratic. We just showed that it is implied by SOC-RLT constraint and $X\succeq xx^\top$ requirement.

Moreover, from $0\preceq (X-xx^\top) + (x-c^i)(x-c^i)^\top = X-x(c^i)^\top-c^ix^\top + c^i(c^i)^\top$ and \eqref{eq:impliedLMI_s} we deduce that
\begin{equation}\label{eq:impliedLMI_w}
    X-x(c^i)^\top-c^ix^\top + c^i(c^i)^\top \preceq r_i^2 I_n.
\end{equation}

In addition, we always have
\begin{align*}
    0\preceq &\begin{bmatrix} x-c^i \\ x-c^j \end{bmatrix}
    \begin{bmatrix} x-c^i \\ x-c^j \end{bmatrix}^\top \\
    &= \begin{bmatrix} xx^\top - x(c^i)^\top - c^ix^\top +c^i(c^i)^\top & xx^\top - x(c^j)^\top - c^ix^\top +c^i(c^j)^\top \\
    xx^\top - x(c^i)^\top - c^jx^\top +c^j(c^i)^\top & xx^\top - x(c^j)^\top - c^jx^\top +c^j(c^j)^\top
    \end{bmatrix}.
\end{align*}
As $xx^\top \preceq X$ and $0 \preceq \begin{bmatrix} 1 & 1 \\ 1& 1\end{bmatrix}$ and Kronecker product of two PSD matrices is PSD as well, we also have
\[
0 \preceq \begin{bmatrix} 1 & 1 \\ 1& 1\end{bmatrix} \otimes (X-xx^\top) = \begin{bmatrix} X-xx^\top & X-xx^\top \\ X-xx^\top & X-xx^\top \end{bmatrix}.
\]
Thus, summing up these two inequalities, we conclude that \begin{align}\label{eq:impliedPSD}
    0\preceq \begin{bmatrix} X - x(c^i)^\top - c^ix^\top +c^i(c^i)^\top & X - x(c^j)^\top - c^ix^\top +c^i(c^j)^\top \\
    X - x(c^i)^\top - c^jx^\top +c^j(c^i)^\top & X - x(c^j)^\top - c^jx^\top +c^j(c^j)^\top
    \end{bmatrix}.
\end{align}
Note that the diagonal blocks in the matrix in \eqref{eq:impliedPSD} are of form $X - x(c^i)^\top - c^ix^\top +c^i(c^i)^\top$, and by \eqref{eq:impliedLMI_w}  these block matrices are dominated in the PSD sense by $r_i^2 I_n$. Therefore, we conclude that inequality~\eqref{eq:Zhen} is satisfied whenever $Z\in\cC_{\{i,j\}}$ and  $Z_{n+2,n+2}=1$.
\end{proof}

\section{The moment-SOS viewpoint}
Here we describe another way to derive $\cC_m$ as a particular case of moment-SOS hierarchy. 

The moment-SOS hierarchy, introduced by Lasserre \cite{lasserre2001global} and Parrilo \cite{parrilo2003semidefinite}, is a well known and studied systematic way of obtaining semidefinite relaxations of increasing tightness and size. See also \cite{laurent2009sums} for a comprehensive introduction. 

Here we briefly describe the idea for the moment hierarchy. For a ground set $\widehat{S}:=\{x:~ g_i(x)\ge 0,\,i\in[k]\}$ defined by $k$ polynomial inequalities $ g_i(x)\ge 0$, $i\in[k]$, the variable in the moment hierarchy is a linear functional on some finite dimensional subspace of polynomials, usually polynomials up to a fixed degree $2d$. Computationally, this linear functional is encoded by the so-called \emph{pseudomoments}, which are the values of this linear functional on a basis of the subspace of polynomials, usually chosen to be the monomial basis. These pseudomoments on $\widehat{S}$ are required to share certain properties of true moments on $\widehat{S}$, which send $f$ to $\int f\, d\mu$ for some nonnegative measure $\mu$ on $\widehat{S}$. Namely, the pseudoments of provably nonnegative polynomials should be nonnegative, which include global sum-of-squares, product of constraint $g_i$ with sum-of-squares (Putinar type), and product of the constraints $g_ig_j$ with possibly other constraints or sum-of-squares (Schm{\"u}dgen type). The pseudomoments can be conveniently arranged into the form of a pseudomoment matrix, whose rows and columns are indexed by polynomials $f_1,\ldots,f_N$, and the $(f_i,f_j)$ entry of the matrix is the pseudomoment of $f_if_j$. 

Going back to our setting, the constraints $g_i(x)\ge 0$ are the ball constraints $r_i^2-\|x-c^i\|_2^2 \ge 0$. The moment matrix will have size $(n+2)\times (n+2)$, indexed by monomials of degree at most 1, plus one more
term $\|x\|_2^2=\sum_i x_i^2$. Same as the usual moment hierarchy, the $(f_i,f_j)$ entry of the pseudomoment matrix models the pseudoment $\cM(f_if_j)$. Specifically, we write our matrix variable $Z$ as: 
\[
Z=\begin{bmatrix}
    \cM(x_1^2) & \ldots & \cM(x_1 x_n) &  \cM(x_1\sum_i x_i^2) & \cM(x_1)\\
     \vdots & \ddots & \vdots & \vdots & \vdots\\
     \cM(x_1 x_n) & \ldots & \cM(x_n^2) & \cM(x_n\sum_i x_i^2) & \cM(x_n)  \\
     \cM(x_1\sum_i x_i^2) & \ldots & \cM(x_n\sum_i x_i^2) & \cM((\sum_i x_i^2)^2) & \cM(\sum_i x_i^2)\\
     \cM(x_1) & \ldots & \cM(x_n) & \cM(\sum_i x_i^2) & \cM(1)
\end{bmatrix}
\]

Now we derive constraints on $Z$ as a pseudomoment matrix as follows: 

\begin{enumerate}
    \item $Z\succeq 0$, as $\cM(\sigma^2)\ge 0$ for all $\sigma$ that is a linear combination of $1,x_1,\ldots,x_n,\sum_i x_i^2$;
    \item $\ip{Q,Z}=0$, as by linearity $\sum_i \cM(x_i^2)=\cM(\sum x_i^2)$;
    \item $(d^i)^\top Z d^j\ge 0$ for all $i,j\in[m]$, as $\cM(g_i(x)g_j(x))\ge 0$; 
    \item $PZd^i\in \L^{n+2}$ follows from that $\cM(g_i(x)\sigma^2)\ge 0$ for $\sigma$ of the form $ax_j+b$, $a,b\in \mathbb{R},j\in [m]$. This is equivalent to $\begin{bmatrix}
    \cM(x_j^2 g_i(x)) & \cM(x_j g_i(x))\\
    \cM(x_j g_i(x)) &\cM(g_i(x))
\end{bmatrix}\succeq 0$ for all $j\in [m]$, which then implies

\[
P\begin{bmatrix}
    \cM(x_1 g_i(x))\\ \vdots \\ \cM(x_n g_i(x)) \\ \cM(g_i(x)\sum_j x_j^2) \\ \cM(g_i(x))
\end{bmatrix}\in \L^{n+2},
\]

and this is exactly $PZd^i\in \L^{n+2}$. 
\end{enumerate}

Thus we have arrived at Burer's lifted formulation in \cite{burer2024slightly} using moment hierarchy. This also shows that the lifted formulation is implied by the full second-level moment hierarchy.

\section{Conclusion and further remarks}

In this paper, we show that Burer's lifted convex relaxation for $S$ is provably tighter than the Kronecker RLT inequalities as well as Zhen et al.'s RLT inequalities. 

From a more theoretical perspective, we show that Burer's lifted relaxation can be interpreted as a particular case of moment-SOS hierarchy (to be more precise, a relaxation of the second-level of moment-SOS hierarchy). The first-level of the moment-SOS hierarchy is known to be the same as the Shor relaxation. While the usual setup of moment hierarchy is known to be computationally inefficient even for the second level, Burer's lifted convex relaxation increases the matrix size by only one compared to the first level hierarchy, and yet seemingly retains most of the strength of the second-level hierarchy, as suggested by results of this paper and \cite{burer2024slightly}. In general, it will be interesting to study the relationship between various RLT type inequalities and the moment-SOS hierarchy. 

\section*{Statements and Declarations}
The authors declare that they have no competing interests.

\section*{Acknowledgments}
This research is supported in part by AFOSR grant FA95502210365 and UKRI Horizon Europe underwrite EP/X032051/1. 

\bibliographystyle{plainnat} 
\bibliography{bib}

\begin{thebibliography}{17}
\providecommand{\natexlab}[1]{#1}
\providecommand{\url}[1]{\texttt{#1}}
\expandafter\ifx\csname urlstyle\endcsname\relax
  \providecommand{\doi}[1]{doi: #1}\else
  \providecommand{\doi}{doi: \begingroup \urlstyle{rm}\Url}\fi

\bibitem[Anstreicher(2017)]{anstreicher2017kronecker}
Kurt~M Anstreicher.
\newblock Kronecker product constraints with an application to the two-trust-region subproblem.
\newblock \emph{SIAM Journal on Optimization}, 27\penalty0 (1):\penalty0 368--378, 2017.

\bibitem[Argue et~al.(2022+)Argue, Kilinc-Karzan, and Wang]{argue2020necessary}
C.J. Argue, F.~Kilinc-Karzan, and A.~L. Wang.
\newblock Necessary and sufficient conditions for rank-one generated cones.
\newblock \emph{{Math.\ Oper.\ Res.}}, Forthcoming, (arXiv:2007.07433), 2022+.

\bibitem[Bao et~al.(2011)Bao, Sahinidis, and Tawarmalani]{bao2011semidefinite}
X.~Bao, N.~V. Sahinidis, and M.~Tawarmalani.
\newblock Semidefinite relaxations for quadratically constrained quadratic programming: {A} review and comparisons.
\newblock \emph{{Math.\ Program.}}, 129:\penalty0 129, 2011.

\bibitem[Bienstock(2016)]{bienstock2016note}
Daniel Bienstock.
\newblock A note on polynomial solvability of the cdt problem.
\newblock \emph{SIAM Journal on Optimization}, 26\penalty0 (1):\penalty0 488--498, 2016.

\bibitem[Burer(2015)]{burer2015gentle}
S.~Burer.
\newblock A gentle, geometric introduction to copositive optimization.
\newblock \emph{{Math.\ Program.}}, 151:\penalty0 89--116, 2015.

\bibitem[Burer and Anstreicher(2013)]{burer2013second}
S.~Burer and K.~M. Anstreicher.
\newblock Second-order-cone constraints for {Extended Trust-Region Subproblems}.
\newblock \emph{{SIAM J.\ Optim.}}, 23\penalty0 (1):\penalty0 432--451, 2013.

\bibitem[Burer(2024)]{burer2024slightly}
Samuel Burer.
\newblock A slightly lifted convex relaxation for nonconvex quadratic programming with ball constraints.
\newblock \emph{Mathematical Programming}, pages 1--23, 2024.

\bibitem[Burer and Yang(2015)]{burer2015trust}
Samuel Burer and Boshi Yang.
\newblock The trust region subproblem with non-intersecting linear constraints.
\newblock \emph{Mathematical Programming}, 149\penalty0 (1-2):\penalty0 253--264, 2015.

\bibitem[Hildebrand(2016)]{hildebrand2016spectrahedral}
R.~Hildebrand.
\newblock Spectrahedral cones generated by rank 1 matrices.
\newblock \emph{{J.\ Global Optim.}}, 64:\penalty0 349--397, 2016.

\bibitem[Kelly et~al.(2022)Kelly, Ouyang, and Yang]{kelly2022note}
Sarah Kelly, Yuyuan Ouyang, and Boshi Yang.
\newblock A note on semidefinite representable reformulations for two variants of the trust-region subproblem.
\newblock \emph{Manuscript}, 2022.

\bibitem[Lasserre(2001)]{lasserre2001global}
Jean~B Lasserre.
\newblock Global optimization with polynomials and the problem of moments.
\newblock \emph{SIAM Journal on optimization}, 11\penalty0 (3):\penalty0 796--817, 2001.

\bibitem[Laurent(2009)]{laurent2009sums}
Monique Laurent.
\newblock Sums of squares, moment matrices and optimization over polynomials.
\newblock In \emph{Emerging applications of algebraic geometry}, pages 157--270. Springer, 2009.

\bibitem[Parrilo(2003)]{parrilo2003semidefinite}
Pablo~A Parrilo.
\newblock Semidefinite programming relaxations for semialgebraic problems.
\newblock \emph{Mathematical programming}, 96:\penalty0 293--320, 2003.

\bibitem[Shor(1990)]{shor1990dual}
N.~Z. Shor.
\newblock Dual quadratic estimates in polynomial and boolean programming.
\newblock \emph{Ann.\ Oper.\ Res.}, 25:\penalty0 163--168, 1990.

\bibitem[Sturm and Zhang(2003)]{sturm2003cones}
J.~F. Sturm and S.~Zhang.
\newblock On cones of nonnegative quadratic functions.
\newblock \emph{{Math.\ Oper.\ Res.}}, 28\penalty0 (2):\penalty0 246--267, 2003.

\bibitem[Yang et~al.(2018)Yang, Anstreicher, and Burer]{yang2018quadratic}
B.~Yang, K.~Anstreicher, and S.~Burer.
\newblock Quadratic programs with hollows.
\newblock \emph{{Math.\ Program.}}, 170:\penalty0 541--553, 2018.

\bibitem[Zhen et~al.(2021)Zhen, de~Moor, and den Hertog]{zhen2021extension}
Jianzhe Zhen, Danique de~Moor, and Dick den Hertog.
\newblock An extension of the reformulation-linearization technique to nonlinear optimization.
\newblock \emph{Available at Optimization Online}, 2021.

\end{thebibliography}

\appendix

\section{Proof of Lemma~\ref{lem:compare_kron}}
     
        For any $i,j\in [n+1]$ we have
        \begin{align*}
            & \begin{bmatrix}
            v^i \\ b_i
        \end{bmatrix}^{\top}\cA_1(x)\begin{bmatrix}
            v^j \\ b_j
        \end{bmatrix}\\
        & = r_1 \ip{v^i,v^j}+b_i\ip{v^j,x-c^1}+b_j\ip{v^i,x-c^1}+r_1b_ib_j\\
        & = \frac{1}{r_1}\left( \ip{r_1 v^i+b_i(x-c^1),\, r_1v^j+b_j(x-c^1)} +b_ib_j( r_1^2-\|x-c^1\|_2^2)\right).
        \end{align*}

        If $\|x-c^1\|_2=r_1$, then we have
        \begin{align*}
            & \cV^{\top}\left(\cA_2(x)\otimes \cA_1(x)\right)\cV\\
    &= r_2\sum_{i=1}^{n+1} \begin{bmatrix}
            v^i \\b_i
        \end{bmatrix}^{\top}\cA_1(x)\begin{bmatrix}
            v^i \\b_i
        \end{bmatrix}+2\sum_{i=1}^n (x_i-c^2_i)\begin{bmatrix}
            v^i \\b_i
        \end{bmatrix}^{\top}\cA_1(x)\begin{bmatrix}
            v^{n+1} \\b_{n+1}
        \end{bmatrix}\\
        &= \frac{1}{r_1r_2}\sum_{i=1}^{n+1}  r_2^2\|r_1v^i+b_i(x-c^1)\|_2^2\\
        &\quad +\frac{2}{r_1r_2}\sum_{i=1}^n r_2(x_i-c^2_i)\ip{r_1 v^i+b_i(x-c^1),r_1v^{n+1}+b_{n+1}(x-c^1)}\\
        &= \frac{1}{r_1r_2}(r_2^2-\|x-c^2\|_2^2) \|r_1v^{n+1}+b_{n+1}(x-c^1)\|_2^2 \\
        &\quad + \frac{1}{r_1r_2}\sum_{i=1}^n \|r_2(r_1v^i+b_i(x-c^1))+(x_i-c^2_i)(r_1v^{n+1}+b_{n+1}(x-c^1))\|_2^2\\
        &\ge  \frac{1}{r_1r_2}(r_2^2-\|x-c^2\|_2^2) \|r_1v^{n+1}+b_{n+1}(x-c^1)\|_2^2,
        \end{align*}

        which proves the desired statement. The proof for the case $\|x-c^1\|<r_1$ and $\|x-c^2\|<r_2$ is very similar, where we in addition need to take care of the difference terms $b_ib_j( r_1^2-\|x-c^1\|_2^2)$. In this case we have
        \begin{align*}
            & \cV^{\top}\left(\cA_2(x)\otimes \cA_1(x)\right)\cV\\
    &= \frac{1}{r_1r_2}(r_2^2-\|x-c^2\|_2^2) \|r_1v^{n+1}+b_{n+1}(x-c^1)\|_2^2 \\
        &\quad + \frac{1}{r_1r_2}\sum_{i=1}^n \|r_2(r_1v^i+b_i(x-c^1))+(x_i-c^2_i)(r_1v^{n+1}+b_{n+1}(x-c^1))\|_2^2\\
        &\quad +\frac{r_1^2-\|x-c^1\|_2^2}{r_1}\left(\sum_{i=1}^{n+1} r_2b_i^2+2\sum_{i=1}^n (x_i-c^2_i) b_ib_{n+1}\right).
        \end{align*}

        We have $r_1^2-\|x-c^1\|_2^2> 0$, and
        \begin{align*}
        &\sum_{i=1}^{n+1} r_2b_i^2+2\sum_{i=1}^n (x_i-c^2_i) b_ib_{n+1}\\
        &=\frac{1}{r_2}\left( \left(r_2^2-\|x-c^2\|_2^2 \right)b_{n+1}^2+\sum_{i=1}^n \left(r_2b_i+(x_i-c^2_i)b_{n+1}\right)^2\right)
        >0. 
        \end{align*}
        Hence, the result follows.

\end{document}